\documentclass[12pt,a4paper,reqno]{amsart}
\usepackage[all]{xy}
\usepackage{amsmath}
\usepackage{amsfonts}
\usepackage{amssymb}
\usepackage{amscd}
\usepackage{color}

\theoremstyle{plain}
\newtheorem{theorem}{Theorem}[section]
\newtheorem{mytheorem}{Theorem}[subsection]
\newtheorem{corollary}[mytheorem]{Corollary}
\newtheorem{lemma}[mytheorem]{Lemma}
\newtheorem{proposition}[mytheorem]{Proposition}
\newtheorem{definition}[mytheorem]{Definition}
\newtheorem{example}[mytheorem]{Example}
\newtheorem{remark}[mytheorem]{Remark}

\theoremstyle{definition}

\def\CC{{\mathbb C}}

\def\FF{{\mathbb F}}

\def\QQ{{\mathbb Q}}
\def\RR{{\mathbb R}}
\def\ZZ{{\mathbb Z}}

\def\Ker{\mathrm{Ker}\,}%

\def\hfl#1#2{\smash{\mathop{\hbox to
10mm{\rightarrowfill}}\limits^{\scriptstyle#1}_{\scriptstyle#2}}}
\def\hflrev#1#2{\smash{\mathop{\hbox to
10mm{\leftarrowfill}}\limits^{\scriptstyle#1}_{\scriptstyle#2}}}
\def\hflcourte#1#2{\smash{\mathop{\hbox to
3mm{\rightarrowfill}}\limits^{\scriptstyle#1}_{\scriptstyle#2}}}
\def\hflrevcourte#1#2{\smash{\mathop{\hbox to
3mm{\leftarrowfill}}\limits^{\scriptstyle#1}_{\scriptstyle#2}}}

\newcommand{\ilim}{\varprojlim_n}

\newcommand{\cal}{\mathcal}

\newcommand{\limp}{\displaystyle{\lim_{\longleftarrow}}}

\DeclareMathOperator{\coker}{coker}

\DeclareMathOperator{\de}{\mathrm{deg}}

\DeclareMathOperator{\Frob}{\mathrm{Frob}}

\newcommand{\QQbar}{\overline{\QQ}}

\def\+{{\dagger}}

\begin{document}
\title[ ] {Tamagawa Number formula with coefficients over varieties in positive characteristic}

\author[Trihan-Brinon]{Fabien Trihan and Olivier Brinon}
\address{Sophia University, Japan}
\email{f-trihan-52m@sophia.ac.jp}
\address{Universite Bordeaux 1}
\email{olivier.brinon@math.u-bordeaux1.fr}

\begin{abstract} We express the order of the pole and the leading coefficient of the $L$-function of a (large class of) $\ell$-adic coefficients ($\ell$ any prime) over a quasi-projective variety over a finite field of characteristic $p$. This is a generalization of the result of \cite{MR} with coefficients. The new key ingredient is the use of $F$-gauges and their equivalence in the derived category with Raynaud modules proved by Ekedahl. 
\end{abstract}

\maketitle \tableofcontents

\begin{section}{Introduction} Let $V$ be a variety of dimension $n$ over a finite field $k=\FF_q$, $q=p^a$, with $p$ a prime and $a\in \ZZ_{>0}$. We assume that $V$ is equipped with an embedding $j:V\hookrightarrow V'$ into a proper variety $V'/\FF_q$. Let $E$ be an overconvergent $F$-isocrystal over $V$ pure, satisfying some integrality condition and with finite cohomological dimension (see Section \ref{coeff} for the precise conditions). We give an expression of the order of the pole and of the leading coefficient of the $L$-function of this coefficients at any integer under some classical semisimplicity conjecture. This is a generalization of the work of Milne-Ramachandran who treat the case of the trivial coefficient and is usually called the Tamagawa Number formula in the number field case. We also give the analogous formula for the $l$ adic companion of such coefficient in the sense of Deligne (see \cite{AE}). The proof is based on the celebrated work of Milne-Ramachandran with as new ingredient, the use of $F$-gauges and their equivalence in the derived category with the Raynaud-modules as proved by Ekedhal (\cite{Ek}). The plan of the paper is as follows: Section 3 is devoted to the setting. In section 4, we show how to associate to a $F$ log crystal over a proper log smooth base a structure of coherent complex of $R$ modules in the derived category. Section 5 is a recall of the  paper of Milne-Ramachandran (\cite{MR}) whose main result is to express the order of the pole and of the leading coefficient of the Zeta function of a coherent complex of $R$-modules. Finally, in the final section we apply the main result of Milne -Ramachandran to the structure of coherent $R$-complex constructed in section 4 via the use of proper hypercovering and proper descent of the rigid cohomology of our nice overconvergent $F$-isocrystal (see Cor. 6.1.3 for our main result). We also give the analogous formula for the $l$-adic companion of our coefficient, when it exists (Th. 6.2). 
\end{section}

\begin{section}{Acknowledgement} We would like to express our gratitude to Kazuya Kato and Atsushi Shiho for very enlightning discussions about the appropriate category of coefficients for our result.
\end{section} 

\begin{section}{Setting} 

\subsection{The variety}\label{variety} Let $V$ be a variety of dimension $n$ over a finite field $k=\FF_q$, $q=p^a$, with $p$ a prime and $a\in \ZZ_{>0}$. We assume that $V$ is equipped with an embedding $j:V\hookrightarrow V'$ into a proper variety $V'/\FF_q$. We denote $i:Z:=V'\setminus V\hookrightarrow V'$.  By \cite{Nakk12}, there is a simplicial proper hypercovering $(j_.:V_.\hookrightarrow V'_.)$ of $(V,V')$ such that $V'_./\FF_q$ is a proper smooth scheme and $V_./\FF_q$ is the complement of a simplicial strict divisor with normal crossings $Z_.$ on $V'_.$, such that $Z_.\to Z$ is a proper hypercovering of $Z$. We denote $V'^\sharp$ the log scheme whose log-structure is the one induced by $Z_.$. $V'^\sharp$ is a proper log-smooth variety over $\FF_q$ endowed with the trivial log-structure.

\subsection{The coefficient}\label{coeff} Let $W$ denote the Witt vectors of $\FF_q$ and $K$ its fraction field. We denote $\bar k$ an algebraic closure of $k$, $\bar{W}=W(\bar{k})$ and $\bar K$ its fraction field. For any log-scheme $X^\sharp/\FF_q$, We denote $F-Cryst(X^\sharp/W)$ the category of finite locally free non-degenerated $F$-crystals over $X^\sharp/W$ for the \'etale topology (see for example \cite{HK}) and $F-iso(X^\sharp/K)$ the category of $F$-isocrystal over $X^\sharp/K$ (see for example \cite{Sh}). 
We denote $F_E$ the $\sigma$-linear endomorphism of $E$ $F_E:=F\circ\tau$, where $\tau:E\to \sigma_{X^\sharp}^*E$ is the map sending $x$ to $x\otimes 1$. For a general variety $X/\FF_q$, we denote $F-iso^\dagger(X/K)$ the category of overconvergent $F$-isocrystal over $X$ (see for example \cite{LS}). Let $E\in F-iso^\dagger(V/K)$. Then by the semistable theorem of Kedlaya (\cite{Ke}), we can choose a finite generically \'etale hypercovering $f_{V_.}:V_.\to V$ as in the previous paragraph and a family $(E^{log}_.)$ of objects of $F-iso(V^{'\sharp}_./K)$ such that 
$$j_.^\dagger E^{log}_.=f_{V_.}^*E$$
where the functor
$$j_.^\dagger: F-iso(V^{'\sharp}_./K)\to F-iso^\dagger(V/K)$$
is the one constructed in \cite{LST}.

We will need the following assumptions:

We will say that the data $(E, V.,E^{log}_.)$ is {\bf (NICE)} if it satisfies the following conditions:
\begin{enumerate}
\item The family of log $F$-isocrystals $(E^{log}_.)$ is integral in the sense that they come from objects of $F-Cryst(V^{'\sharp}_./W)$ (Recall that we have a canonical functor from the category of crystals to isocrystals).
\item $H^i(V_.,E)=0$ for $i>N(E)$, for some constant $N(E)$ independant of $i$.
\end{enumerate}

\begin{remark} The second hypothesis is satisfied for example in the following cases:
\begin{enumerate}
\item The variety $V$ is smooth, in which case $N(E)=2n$.
\item The overconvergent $F$-isocrystal $E$ has constant Frobenius slope. In this situation, it is possible to define on the de Rham-Witt cohomology of $E^{log}$ a slope filtration and the argument of \cite{Nakk12}, 11..7.5 will work with coefficient.
\end{enumerate}
\end{remark}






\subsection{L-function}\label{Lfct} The $L$-function of $E$ is defined as

$$L(V,E,t):=\prod_{v\in V} det(1-tF_v, E_v)^{-1}$$

where $E_v$ is the fiber of $E$ at the closed point $v$ of $V$, an $F$-isocrystal on $k(v)/K$. It was proved in \cite{ELS} that this infinite Euler product is a rational function in $K(t)$ that can be expressed as follows:

$$L(V,E,t)=\prod_i det(1-t\Phi^i_E, H^i_{rig,c}(V/K,E))^{(-1)^{i+1}},$$
where $\Phi^i_E:=H^i(F_E^a)$.

We will need the following second assumption on the coefficient. We denote $\iota$ the choice of an embedding  $\QQ_p\in \CC$.

\noindent ($\iota$-{\bf PURE}) For any $v$, the eigenvalues $\alpha$ of the Frobenius acting on $E_v$ are Weil numbers , i.e. such that $|\iota(\alpha)|=q^{w/2}$ for some integer $w$ independent of $v$.  

Under this hypothesis, we recall that it was proved in \cite{Ked-Weil2} that the $F$-isocrystals over $\FF_q$ $H^i_{rig,c}(V/K,E)$ are $\iota$-mixed, in particular the eigenvalues of $F^a$ are Weil numbers. 

\end{section} 

\begin{section}{$F$-gauge complex of a $F$-log crystal}

In the next two subsections, we recall the construction of \cite{Ek}, p.36-37.

\subsection{From $F$-log crystals to virtual crystals}\label{F-log}  Let $X/\FF_q$ be a proper smooth variety, $Z$ a divisor with normal crossing with respect to $X$. We denote $X^\#/\FF_q$ the associated proper smooth log-variety. Let $E$ be a $F$-crystal on $X^\#/\FF_q$.

\begin{definition} A virtual $F$-crystal (resp. of finite type) is a triple $(U,\underbar F,N)$ where $U$ is a finite dimensional $K$-vector space, $\underbar F$ a $\sigma$-linear automorphism of $U$ and $N$ a (resp. finitely generated) $W$-submodule of $U$ such that $N\otimes K=U$.
\end{definition}

\begin{example} We consider on $U(E,i):=H^i_{crys}(X^\#/W,E)\otimes\QQ_p$ the $\sigma$-linear map:
$${\underbar F}^i_E:=H^i(F_E)=H^i(F)\circ H^i(\tau): U(E,i)\to U(\sigma^*E,i)\to U(E,i).$$
By definition, $H^i(F)$ is clearly a $\sigma$-linear bijective map. By \cite{ELS}, Prop. 2.1, the map $H^i(\tau^a)\otimes\QQ_p:U(E,i)\to U(\sigma^{a,*}E,i)$ is a bijection. This map is the composition of $\sigma$-linear maps induced by the adjonction map of the absolute Frobenius on $X^\#$:
$$U(E,i)\to U(\sigma^*E,i)\to\dots\to U(\sigma^{a,*}E,i).$$ We claim that the first one, $H^i(\tau)$ is also bijective. Indeed, since the composed is injective, this map (as well as all the following maps) is also injective. The surjectivity of the composed map  and the injectivity of the map $U(\sigma^*E,i)\to U(\sigma^{a,*}E,i)$ implies then the surjectivity of $H^i(\tau)$. We denote ${\bf VC}(E,i)$ the virtual crystal $(H^i_{crys}(X^\#/W,E)\otimes\QQ_p,{\underbar F}^i_E, N(E,i):= H^i_{crys}(X^\#/W,E))$.
\end{example}

\subsection{From virtual crystal to $F$-gauges} 

\begin{definition} A $F$-gauge $(M, \tilde F,\tilde V,\tau)$ is the data of a graded $W$-module $M=\otimes_{i\in\ZZ} M^i$ endowed with linear mappings of respective degree 1 and -1 $\tilde F$ and $\tilde V$ and a $\sigma$-linear isomorphism $$\tau:M^\infty:=\varinjlim (\dots\to M^i\buildrel{\tilde F}\over\to M^{i+1}\to\dots)\to M^{-\infty}:=\varinjlim (\dots\to M^i\buildrel{\tilde V}\over\to M^{i-1}\to\dots).$$ 
\end{definition}

Ekedahl defines in \cite{Ek}, p.37 a functor ${\bf Hodge}(.)$ from the category of virtual crystals to the category of $F$-gauges as follows: Let $(U,\underbar F, N)$ be a virtual crystal. Then, setting $M^i:=\underbar F^{-1}(p^iN)\cap N$, we get a filtration $\{M^i\}$ of $N$ satisfying
\begin{enumerate}
\item[(i)] $pM^i\subset M^{i+1}$.
\item[(ii)] $N=\cup_i M^i$.
\item[(iii)] $\underbar F$ maps $\cup_i p^{-i}M^i$ into and onto $\cup M^i$.
\end{enumerate}
Set $M:=\otimes_{i\in\ZZ} M^i$, $\tilde F$ the multiplication by $p$, $\tilde V$ the inclusion and $\tau$, the map deduced by the property (iii). Then ${\bf Hodge}(U,\underbar F, N):=(\otimes_{i\in\ZZ} M^i, \tilde F,\tilde V,\tau)$ is a F-gauge. 

Let $E$ be a $F$-crystal on $X^\#/\FF_q$. We denote ${\bf FG}(E,i)$, the $F$-gauge ${\bf Hodge}({\bf VC}(E,i))$ and ${\bf FG}(E)$ the complex of $F$-gauges whose degree $i$-term is ${\bf FG}(E,i)$ and derivation is the zero map.

\subsection{Coherent object in the derived category of $F$-gauges}

\begin{definition} An object $M$ in the derived category of $F$-gauges, $D(F-g)$ is called coherent if the canonical map $M\to \RR\ilim W/p^n\otimes^L_W M:=\hat M$ is an isomorphism and if $\FF_q\otimes^L_W M$ is a coherent $\FF_q$-complex. We denote $D^b_c(F-g)$ the derived category of bounded complexes of coherent $F$-gauges.    
\end{definition} 

\begin{theorem} Let $E$ be a $F$-crystal on $X^\#/\FF_q$. Then $${\bf FG}(E)\in 
D^b_c(F-g).$$
\end{theorem}

\begin{proof} First note that ${\bf FG}(E)$ can be seen as a bounded complex of $F$-gauges concentrated in degree 0 to $2dim(X)$ and in particular as an object of $D^b(F-g)$.  By \cite{Ek}, Prop. 5.2, it is enough to check that $H^j({\bf FG}(E))^i={\bf FG}(E,j)^i$ is a finitely generated $W$-module for any $0\leq j\leq 2dim(X)$ and $0\leq i\leq N$, which is obvious since ${\bf FG}(E,j)^i$ is a sub $W$-module of the finitely generated $W$-module $H^j(X^\#/W,E)$.
\end{proof}

\subsection{Coherent object in the derived category of Raynaud modules associated to a $F$-log crystal}\label{Dbc}

\begin{definition}
We denote $R$ the $W$-graded algebra $R^0\oplus R^1$ generated by $F$ and $V$ in degree 0 and by $d$ in degree 1, subject to the relation:\\
$FV=p=VF$, $Fa=\sigma(a)F$, $aV=V(\sigma(a))$, $d^2=0$, $FdV=pd$, $Vd=pdV$, $dF=pFd$, $da=ad$ ($a\in W$). 
\end{definition}

We can see a graded $R$-module $M$ as a complex
$$\dots \to M^i\buildrel{d}\over\to M^{i+1}\buildrel{d}\over\to\dots$$
such that for any $i$, $M^i$ is a $R^0$-module endowed with a differential $d$ such that $FdV=pd$. We define the $n-$th degree shift $M\{n\}$ of $M$ as the graded $R$-module whose $i-th$ degree is $M^{n+i}$ and derivation is $(-1)^n d_M^{n+i}$.

\begin{definition}
\begin{enumerate}
\item We say that a $R$-module is elementary of Type I if it is a $R^0$-module, finitely generated over $W$ such that $V$ is topologically nilpotent. Such module can be written as the direct sum of a free $W$-module of finite rank with an action of $F$ such that the slopes are in $[0,1[$ and of a torsion part which is of finite length as $W$-module. The torsion part is an iterated extension of the module $(k,F,V=0)$ or ($k, F=\sigma,0)$.
$$\prod_{n\geq 0} kV^n\buildrel{d}\over\to \prod_{n\geq l} kdV^n$$
where $l\in \ZZ$ and $dV^n=F^{-n}d$ if $n<0$.
\item A graded $R$-module $M$ is said coherent if it admits a finite filtration whose quotients are degree shifts of elementary modules of type I or II. 
\item A complex $M$ of $R$-modules is said to be coherent if it is bounded with coherent cohomology. We denote $D^b_c(R)$ the full subcategory of $D(R)$ consisting of coherent complexes.
This is a triangulated subcategory of $D(R)$; in particular, the coherent modules form an abelian subcategory of $Mod(R)$ closed under extensions (\cite{I}, 2.4.8).
We denote $D^b_c(R)$ the full subcategory of $D(R)$ consisting of coherent complexes.
\end{enumerate}
\end{definition}

The main theorem of \cite{Ek} is the following result. 

\begin{theorem}\label{ek-main}(\cite{Ek}, Th. 5.3) There is an equivalence of categories:
$${\bf R}:D^b_c(F-g)\to D^b_c(R).$$
\end{theorem}
We denote $\bf S$ the inverse functor. Coming back to the hypothesis of section \ref{F-log}, let $E$ be a (finite locally free) $F$-crystal on a proper log-smooth variety $X^\#/\FF_q$ and denote ${\bf R}(E)$ the object of $D^b_c(R)$:
$${\bf R}(E):={\bf R}({\bf FG}(E)).$$

\subsection{Operations in $D^b_c(R)$}

\subsubsection{Translation} We can see a graded $R$-module $M$ as a complex
$$\dots \to M^i\buildrel{d}\over\to M^{i+1}\buildrel{d}\over\to\dots$$
such that for any $i$, $M^i$ is a $R^0$-module endowed with a differential $d$ such that $FdV=d$. We define the $n-$th translated $M\{n\}$ as the graded $R$-module whose $i-th$ degree is $M^{n+i}$ and derivation is $(-1)^n d_M^{n+i}$.

\subsubsection{Tate Twist} We can see a complex of graded $R$ module $M$ as a bicomplex $M^{i,j}$ where $M^{.,j}$ is a graded $R$-module and $M^{i,.}$ is a complex of $R^0$-module. For an integer $r$ and $M$ a complex of graded $R$-modules, we denote $M(r)$ the complex of graded $R$-modules whose associated bicomplex has entries $M(r)^{i,j}:=M^{i+r,j-r}$ and with the appropriate sign changes on the differentials.

\subsubsection{Simple complex} Let $M$ be a complex of graded $R$-module, we can associate to $M$ a simple complex of $W$-modules $sM$ such that 
$$sM^n:=\oplus_{i+j=n} M^{i,j}, d x^{i,j}= d'x^{i,j}+(-1)^i d"x^{i,j}.$$
This construction extends to give a functor
$$s:D^+(R)\to D(W).$$ If we restrict this functor to the category $D^b_c(R)$, then $sM$ is then a perfect complex of $W$-modules.

\subsubsection{Slope spectral sequence} Let $M$ be a complex of graded $R$-module, then $H^j(M)$ is a graded $R$-module with $i$-th degree:
$$(H^j(M))^i=H^j(M^{i,.}).$$

We have a spectral sequence

$$E_1^{i,j}:=(H^j(M))^i\Rightarrow E^{i+j}=H^{i+j}(sM)$$ called the slope spectral sequence

\begin{theorem} (\cite{I}, Corollary 2.5.4) Let $M\in D^b_c(R)$. The slope spectral sequence of $M$ degenerates at $E_1$ modulo torsion and $E_2$ modulo $W$-modules of finite length.
\end{theorem}

\begin{corollary}(\cite{MR}, (9), (10))\label{slopes} 
Let $M$ be a complex of graded $R$-module with only non-negative degrees and let $F'$ acts on $M^{i,j}$ as $p^{i}F$. Using the relation $Fdv=pd$ for the horizontal differential and the compatibility with $F$ for the vertical one, one can check that both differentials of $M^{i,j}$ commutes with $F'$ and we have an isomorphism of $F$-isocrystals over $k$:
$$((H^j(M))^i\otimes K,p^{i}F)=(H^{i+j}(sM)\otimes K,F')_{[i,i+1[},$$

where for an isocrystal $(M_K,F)$, we denote $(M_K,F)_{[i,i+1[}$ the piece of $M_K$ where $F'$ acts with slopes in $[i,i+1[$.
\end{corollary}

\begin{lemma}(\cite{Ek}, VII, 3.) Let $M\in D^b_c(R)$. Then the slope spectral sequence of $M$ is the $\infty$-part of the spectral sequence of $F$-gauges
$$E^{i,j}_1={\bf S}(H^j(M)^i)\Rightarrow H^{i+j}({\bf S}(M)).$$
In particular, for any integer $k$, we have
$$H^k(sM)\simeq H^k({\bf S}(M))^\infty.$$
\end{lemma}

\begin{corollary} let $E$ be a (finite locally free) $F$-crystal on a proper log-smooth variety $X^\#/\FF_q$. We have an isomorphism $$H^i(s{\bf R}(E))\simeq H^i_{crys}(X^\#/W,E)$$
compatible with the Frobenius operator. 
\end{corollary}

\begin{proof} By the previous lemma, we have $H^i(s{\bf R}(E))\simeq H^i({\bf S}({\bf R}(E)))^\infty
=H^i({\bf FG}(E))^\infty$ by Th. \ref{ek-main}. Now since the derivation is zero on ${\bf FG}(E)$, we have $H^i({\bf FG}(E))={\bf Hodge}({\bf VC}(E,i))$ and its $\infty$ part is isomorphic, via the morphism $\tau$ given by its structure of $F$-gauge, to its $-\infty$ part, which is $H^i_{crys}(X^\#/W,E)$.
\end{proof}

\subsubsection{Internal Homs and tensor products}
One can construct (\cite{I}, 2.6.1.10) a tensor product bifunctor
$$.*^L.:D^b_c(R)\times D^b_c(R)\to D^b_c(R)$$
for which $\underbar W$ is the unit object. 

We can also construct (see \cite{I},2.6.2) an internal Hom bifunctor:
$$\RR\underbar{Hom}(.,.):D^b_c(R)\times D^b_c(R)\to D^b_c(R).$$

\subsubsection{Homological algebra in $D^b_c(R)$}
For the convenience of the reader, we give a summary of the section 4 of \cite{MR}. We refer the reader to loc. sic for the proofs.\\

\noindent Let $\cal S$ denote the category of sheaves of commutative groups on the \'etale site of perfect schemes over Spec(k), and $\cal G$ denote the abelian subcategory of commutative perfect algebraic group schemes killed by a power of $p$. We will denote ${\cal S}_.$ (resp. ${\cal G}_.$) the category of projective system $(P_m)_m$ of $\ZZ/p^m$-modules in $\cal S$
(resp. in $\cal G$).

We have a functor 
$$.^F:D^B_c(R)\to D^{perf}(\ZZ_p)$$
$$M\mapsto M^F:=\RR Hom({\underbar W},M).$$

\begin{lemma}\label{P^F} Let $M,N\in D^b_c(R)$. We have 
$$\RR\underbar{Hom}(M,N)^F=\RR Hom(M,N).$$
\end{lemma}

We can sheafify the construction as follows: if $M$ is a graded $R$ -module, we denote 
${\cal M}^i_m$ the sheaf in $\cal S$:
$$Spec(A)\mapsto M^i_m\otimes_W W(A),$$
where $M^i_m=R_m\otimes_{R} M^i$.

For $m$ variable, ${\cal M}^i:=({\cal M}^i_m)_m$ is a well-defined of ${\cal S}_.$. Finally, $({\cal M}^.,d,F\otimes\sigma_{W(A)},V\otimes\sigma_{W(A)}^{-1})$ has a structure of $R_.$-module object of ${\cal S}_.$.

We set ${\cal M}_.^F:=Cone({\cal M}_.^0\buildrel{1-F}\over\to {\cal M}_.^0)[-1]$.

We also have a derived functor 

$$\RR\Gamma(Spec(k)_{et},.):D({\cal S}_.)\to D(\ZZ_p),$$
$$(M_m)_m\mapsto \RR\limp_m \RR\Gamma(Spec(k)_{et},.M_m).$$

\begin{theorem}\label{Mi4.6-8}  Let $M\in D^b_c(R)$.

\begin{enumerate}
\item $$\RR\Gamma(Spec(k)_{et}, {\cal M}_.^F)=M^F.$$
\item For any integer $r$ and for any $j$ we have a short exact sequence:
$$0\to U^j\to H^j({\cal M}(r)_.^F)\to D^j\to 0$$
where $U^j$ is a connected unipotent perfect algebraic and $D^j$ a profinite \'etale group.
\item $D^j(\bar k)$ is a finitely generated $\ZZ_p$-module and we have 
$$D^j(\bar k)\otimes\QQ_p\simeq (H^r(H^j(M))\otimes \bar K)^{F\otimes\sigma}.$$
\end{enumerate}

\end{theorem}

In particular, we conclude from the second assertion of the previous theorem:

\begin{corollary} Let $M\in D^b_c(R)$. Then
$${\cal M}_.^F\in D^b_{{\cal G}_.}({\cal S}_.),$$
the derived category of complex of sheaves in ${\cal S}_.$ whose cohomology lies in ${\cal G}_.$. 

\end{corollary}

Applying this corollary to $P=\RR\underbar{Hom}(M,N)$, with $M,N\in D^b_c(R)$, we obtain a functor:

$$\RR{\cal H}om(.,.): D^b_c(R)\times D^b_c(R)\to D^b_{{\cal G}_.}({\cal S}_.)$$
$$(M,N)\mapsto {\cal P}_.^F.$$ 

We deduce from the first assertion of Theorem \ref{Mi4.6-8} applied to $\RR\underbar{Hom}(M,N)$ and from the equality of Lemma \ref{P^F}:

\begin{theorem}\label{Mi4.9} Let $M,N\in D^b_c(R)$.
$$\RR\Gamma(Spec(k)_{et}, \RR{\cal H}om(M,N))=\RR Hom(M,N).$$
\end{theorem}

\end{section}

\begin{section}{The main Theorem of Milne-Ramachandran}
Again, for the comfort of the reader, we state the main result of Milne and Ramachandran and sketch their proof.\\

\noindent Let $\Gamma:=Gal(\bar k/k)$, topologically generated by the Frobenius element $\gamma:x\mapsto x^q$. We have a natural exact functor
$$D^b_c(R)\to D^b_c({\bar R}), M\mapsto \bar M:=M\otimes_W \bar{W},$$ moreover $\bar M$ is endowed with an action of $\Gamma$. For a $\Gamma$-module $M$, we denote $M^\Gamma:=H^0(\Gamma,M)$ and $M_\Gamma=H^1(\Gamma,M)$.
  
   We have a Horschild-Serre spectral sequence:
 
 $$H^i(\Gamma, Ext^j(\bar M, \bar N))=> Ext^{i+j}(M,N).$$
 
 We deduce the following composed map:
 
 $$d^j: Ext^j(M,N)\to Ext^j(\bar M, \bar N)^\Gamma\to Ext^j(\bar M, \bar N)\to Ext^j(\bar M, \bar N)_\Gamma\to Ext^{j+1}(M,N),$$
where the first map comes from the map $E^j\to E^{0,j}$ of the previous spectral sequence, the second map is the canonical inclusion, the third map is the canonical projection and the last map corresponds to the map $E^{1,j}\to E^{j+1}$. 

The same formalism as in \cite{MR}, p.24 shows that $E(M,N):=(Ext^.(M,N),d^.)$ defines a complex. 

\subsection{Statement of the theorem}

Let $M,N\in D^b_c(R)$ and denote $P:=\RR{\underbar Hom}(M,N)$. Recall that $P\in D^b_c(R)$ so that $R_1\otimes^L_{R} P$ is a bounded complex of graded $k$-vector spaces whose cohomology groups have finite dimensions. We denote 

$$h^{i,j}(P):=dim_k H^j(R_1\otimes^L_{R} P)^i,$$ called the Hodge numbers of $P$.

Note also that for any $j$, $H^j(sP)\otimes K$ is an $F$-isocrystal. We denote 

$$Z(P,t):=\prod_j det(1-tF^a, H^j(sP)\otimes K)^{(-1)^{j+1}}.$$

\begin{theorem}(\cite{MR}, Th. 0.1)\label{MR-main} Let $r$ be an integer and assume that $q^r$ is not a multiple root of the minimum polynomial of $F^a$ acting on $H^j(sP)\otimes K$ for any $j$. 

\begin{enumerate}
\item The groups $Ext^j(M,N(r))$ are finitely generated $\ZZ_p$-modules and the alternating sum of their rank is zero. 

\item The Zeta function $Z(P,t)$ of $P$ has a pole at $t=q^{-r}$ of order 
$$\rho:=\sum_j (-1)^{j+1}j.rank_{\ZZ_p}(Ext^j(M,N(r)).$$

\item The cohomology groups of the complex $E(M,N(r))$ are finite and the alternating product of their orders $\chi(M,N(r))$ satisfies
$$| lim_{t \to q^{-r}} Z(P,t) (1-q^rt)^\rho|_p^{-1}=\chi(M,N(r)).q^{\chi(P,r)}$$
where 
$${\chi(P,r)}=\sum_{i,j, i\leq r} (-1)^{i+j}(r-i)h^{i,j}(P)$$
and the $p$-adic valuation $|.|_p$ has been normalized so that $|p^r\frac{m}{n}|_p^{-1}=p^r$ if $m$ an $n$ are prime to $p$.
\end{enumerate}
\end{theorem} 

\subsection{Sketch of proof of the Th. \ref{MR-main}}
We will need some notation: For a group homomorphism $f$ with finite kernel and cokernel, we denote $z(f):=\frac{[\Ker(f)]}{\coker(f)]}$. If $(V,F)$ is an $F$-isocrystal on $\FF_q$, then we will denote
$$V_{(\lambda)}:=\{\bar v\in V\otimes\bar K, F\otimes\sigma(\bar v)=p^\lambda\bar v\}.$$

Note that the eigenvalues of $\gamma$ acting on  $V_{(\lambda)}$ are the $\{\frac{q^\lambda}{\alpha}\}_\alpha$, where the $\alpha$'s run over the eigenvalues of $F^a$ on $V$ satisfying $ord_q(\alpha)=\lambda$.

Let $r$ be an integer, $M,N\in D^b_c(R)$ and let $P:=\RR\underbar{Hom}(M,N)$. We denote $(a_{jl})_l$ the eigenvalues of $F^a$ acting on $H^j(sP)\otimes K$. Recall that according to a theorem of Manin, ($ord_q(a_{jl}))_l$ is the family of slopes of $H^j(sP)\otimes K$.

We will also need the following lemma of \cite{MR}:

\begin{lemma}\label{z(f)} (\cite{MR}, Lemma 5.1)
Let $M$ be a finitely generated $\ZZ_p$-module with an action of $\Gamma$ and let $f:M^\Gamma\to M_\Gamma$ be the map induced by the identity map. Then $z(f)$ is defined if and only if 1 is not a multiple root of the minimal polynomial of $\gamma$ on $M$, in which case $M^\Gamma$ has rank equal to the multiplicity of 1 as an eigenvalue of $\gamma$ acting on $M_{\QQ_p}$ and
$$z(f)=|\prod_{i,a_i\neq 1} (1-a_i)|_p.$$
\end{lemma} 

\begin{proposition}\label{keyprop} Let $r$ be an integer, $M,N\in D^b_c(R)$ and let $P:=\RR\underbar{Hom}(M,N)$. Let also $G^j:=\mathcal{E}xt^j(M,N(r))=H^j(\mathcal{P}(r)_.^F)$ (a perfect algebraic group over $Spec(k)_{et}$). We denote $U^j$ its unipotent part and $D^j$ its pro-\'etale part. Then,\\

$z(f_j:Ext^j(\bar M, \bar N(r))^\Gamma\to Ext^j(\bar M, \bar N(r))_\Gamma)$ is defined if and only if $q^r$ is not a multiple root of the minimal polynomial of $F^a$ acting on $H^j(sP)\otimes K$, in which case, 

$$z(f_j)=|\prod_{a_{jl}\neq q^r}(1-\frac{a_{jl}}{q^r})|_p.|\prod_{ord_q(a_{jl})<r} \frac{q^r}{a_{jl}}|_p. [U^j(k)].$$
\end{proposition}

\begin{proof} We have $G^j(\bar \FF_q)=H^j(\bar{\mathcal{P}}(r)_.^F)(\bar \FF_q)=Ext^j(\bar M,\bar{N}(r))$ by Theorem \ref{Mi4.9}. On the other hand, applying the Theorem \ref{Mi4.6-8}, with $M=P(r)$, we deduce that $D^j(\bar \FF_q)$ is a finitely generated $\ZZ_p$-module and that 

$$D^j(\bar \FF_q)\otimes \QQ_p\simeq [H^0(H^j(P(r))\otimes\bar K]^{F\otimes\sigma}.$$

Since the slope spectral sequence degenerates at $E_2$ after killing the torsion, we can rewrite the right hand side of the previous isomorphism as

$$[H^j(P(r))^0\otimes K]_{(0)}\simeq [H^j(sP(r)\otimes K)_{[0,1[}]_{(0)},$$
by Corollary \ref{slopes}.

Finally, we have,
$$[H^j(sP(r)\otimes K)_{[0,1[}]_{(0)}\simeq [H^j(sP(r)\otimes K)]_{(0)}\simeq [H^j(sP\otimes K)]_{(r)}.$$

In particular, the eigenvalues of $\gamma$ acting on $D^j(\bar k)\otimes\QQ_p$ are the $\frac{q^r}{a_{jl}}$, with $ord_q(a_{jl})=r$.

Looking at the action of $1-\gamma$ on each groups $U^j(\bar k)$, $G^j(\bar k)$ and $D^j(\bar k)$, we deduce by diagram chase, the following commutative diagram of short exact sequences\\

\centerline{
$\begin{matrix}
0&\to&U^j(\bar \FF_q)^\Gamma&\to&G^j(\bar \FF_q)^\Gamma&\to&D^j(\bar \FF_q)^\Gamma&\to&0\\
&&\downarrow f'_j&&\downarrow f_j&&\downarrow f"_j&&\\
0&\to&0&\to&G^j(\bar \FF_q)_\Gamma&\to&D^j(\bar \FF_q)_\Gamma&\to&0
\end{matrix}$}

Using the reasoning above as well as the lemma \ref{z(f)}, we deduce that $z(f"_j)$ is defined if and only if $q^r$ is not a multiple root of the minimum polynomial of $F^a$ acting on $H^j(sP)\otimes K$ and 
$$z(f"_j):=|\prod_{ord_q(a_{jl})=r, a_{jl}\neq q^r}(1-\frac{q^r}{a_{jl}})|_p.$$

Now, we just use that $z(f_j)=z(f'_j)z(f"_j)$ to conclude.
\end{proof}

The proof of Th. \ref{MR-main} is now reduced to the following:

\begin{theorem} Let $M,N\in D^b_c(R)$, $r$ an integer and denote as usual $P=\RR\underbar{Hom}(M,N(r))$ and for any $j$ denote $U^j$, the unipotent part of $H^j(\mathcal{P}(r)_.^F)$. Assume that $q^r$ is not a multiple root of the minimum polynomial of $F^a$ acting on $H^j(sP)\otimes K$ for any $j$. 

\begin{enumerate}
\item The groups $Ext^j(M,N(r))$ are finitely generated $\ZZ_p$-modules and the alternating sum of their rank is zero. 

\item The Zeta function $Z(P,t)$ of $P$ has a pole at $t=q^{-r}$ of order 
$$\rho:=\sum_j (-1)^{j+1}j.rank_{\ZZ_p}(Ext^j(M,N(r)).$$

\item The cohomology groups of the complex $E(M,N(r))$ are finite and the alternating product of their orders $\chi(M,N(r))$ satisfies
$$| \lim_{t \to q^{-r}} Z(P,t) (1-q^rt)^\rho|_p^{-1}=\chi(M,N(r)).q^{\tilde{\chi}(P,r)}$$
where 
$$\tilde{\chi}(P,r)=\sum_j (-1)^{j+1}ord_q([U^j(k)])+\sum_{j,l, \lambda_{jl}\leq r} (-1)^{j+1}(r-\lambda_{jl}),$$
with $\lambda_{jl}$ the slopes of $H^j(sP\otimes K)$.
\end{enumerate}
\end{theorem}

\begin{proof} The first two assertions are proved in \cite{MR}, p.25. We have 
$$Z(P,t)=\prod_{j}(\prod_l(1-a_{jl}))^{(-1)^{j+1}}$$
and we can show exactly as in \cite{MR} that 
$$\chi(M,N(r))=\prod_j z(f_j)^{(-1)^j}.$$

Note also that 
$$|\prod_{a_{jl}\neq q^r}(1-\frac{a_{jl}}{q^r})|_p=|\lim_{t \to q^{-r}} \frac{P_j(t)}{(1-q^rt)^{\rho_j}}|_p,$$

where $\rho_j$ is the multiplicity of $q^r$ as an inverse root of $P_j(t)$ and is equal to $rank Ext^j(\bar M,\bar N(r))^\Gamma$ (or thanks to our hypothesis on $r$ to $rank Ext^j(\bar M,\bar N(r))_\Gamma$)
and that by Manin's Theorem:

$$|\prod_{ord_q(a_{jl})<r} \frac{q^r}{a_{jl}}|_p=\sum_{l, \lambda_{jl}<r} r-\lambda_{jl}.$$

The result follows immediately by Proposition \ref{keyprop}.
\end{proof}

\end{section}

\begin{section}{Application to special values of geometric $L$-function}

We come back to the setting of the first chapter: Let $(V_.,V'_.)$ be the proper hypercovering of $(V,V')$.

\begin{definition} Let $K$ be a finite extension of $\QQ$ and let $S$ be a set of pairs $(\ell,\iota)$ of a prime number $\ell$ and an embedding $\iota : K \to \QQbar_\ell$. A strict compatible $(K,S)$-family of $\ell$-adic sheaves ($\ell$ any prime) corresponds to the data $E=(E_{\ell,\iota})_{(\ell,\iota) \in S}$ of smooth $\QQbar_\ell$-sheaves $E_{\ell,\iota}$ on $V$ ($\ell\neq p$) and of an overconvergent $F$-isocrystal $(E_{p,\iota},\varphi_p)$ over $V/\QQbar_p$ such that for any $x\in V$, $\det(1-\Frob_x q^{-s \cdot \de(x)}; F_{\ell,\iota,x}) = \iota(P_x)$ (where $F_{\ell,\iota,x}$ is defined as above) for some $P_x \in K[q^{-s}]$ which is independent of $(\ell,\iota) \in S$. We will call a strict compatible system nice if the $l$-adic sheaves are $\iota$-pure and coming from $\ZZ_l$-smooth sheaves, and $E_p$ is nice in the previous sense.    
\end{definition}

We give an example of existence of strict compatible system of $\ell$-adic sheaves.

\begin{theorem} Assume that $V$ is smooth and geometrically connected over $\FF_q$. Let $E_p$ be an irreducible overconvergent $F$-isocrystal over $V$ with finite determinant. Then $E_p$ is the p-adic coefficient of a strict compatible system. In particular, 
$$L(E_p,t)\in \Omega(t),$$
for some finite extension $\Omega/\QQ$.
\end{theorem}

\begin{proof} By \cite{AE}, Theorem 4.2, we can associate for any prime $\ell$, an $\ell$-adic sheaf with same $L$-function. The fact that this $L$-function lies in $\Omega(t)$ for some finite extension $\Omega$ of $\QQ$ has been proved by Deligne (see \cite{E-K}).
\end{proof}

Until the end, we assume given a strict compatible system $E$ which is {\bf nice} (in the sense of Section \ref{coeff}) and whose $\ell$-adic sheaves (any $\ell$) are $\iota_\ell$-pure. we will consider an integer $r$ such that $q^r$ is not a multiple root of the minimum polynomial of $\varphi_p^a$ acting on $H^j_{rig,c}(V,E_p)$ for any $j$ nor of the minimum polynomial of $\varphi_\ell$ acting on $H^j_{et,c}(\bar V, E_\ell)$.

Note that by our hypothesis we have 

$$L(V,E_\ell,t)=L(V,E_p,t),$$

for any prime $\ell$. We denote this function $L(V,E,t)$.

\subsection{The $p$-adic valuation of the leading coefficient}

Let $(E^{log}_.)$ some locally free finite rank objects of $F-Cryst((V'_.)^\sharp/W)$ associated to the {\bf nice} $E_p$ . We will denote all this data $\underbar{E}_p$.

For any $j$, we consider ${\bf R}(E^{log}_j)\in D^b_c(R)$, the object defined in Section \ref{Dbc}. As $D^b_c(R)$ is a triangulated subcategory of $D(R)$, we deduce that for any integer $d$,
$${\bf R}(E^{log}_{|V_{\leq d}})\in D^b_c(R)$$
where for any simplicial scheme $S.$, $S_{\leq d}$ denote the truncation on the right at order $d$ of $S_.$.
We denote
$$M(\underbar{E}_p,d):=\tau_{\leq N(E)}{\bf R}(E^{log}_{|V_{\leq d}})$$ the object of $D^b_c(R)$, where $\tau_{\leq *}$ is the $*$-th truncation functor in the derived category. 

\begin{lemma} For $d$ large enough, we have
$$H^i(sM(\underbar{E}_p,d))\otimes\QQ_p\simeq H^i_{rig}(V,E_p)$$
\end{lemma}

\begin{proof}
The inclusion $V'_{\leq d}\to V'_.$, induce an isomorphism  
$$\tau_{\leq N(E)}H^i_{crys}(V'^{\#}_./W,E^{log}_.)\otimes \QQ_p\simeq \tau_{\leq N(E)}H^i_{crys}(V'^{\#}_{\leq d}/W,E^{log}_{\leq d})\otimes \QQ_p=H^i(sM(\underbar{E}_p,d))\otimes\QQ_p$$
for $d$ large enough. Indeed, observe that both sides coincide with $\RR\Gamma_{rig}(V,E_p)$. For the left-hand side, this is a consequence of the proper descent of rigid cohomology with coefficients (observe that the rigid cohomology of $E_p$ over $V$ vanishes for $i$ greater than $N(E)$) and the comparison between rigid and log-crystalline cohomology (\cite{Tr}, Prop. 1.5). For the right-hand side, note that by the second hypothesis of {\bf nice}, a finite number of $V_i$'s suffices to describe the rigid cohomology of $E_p$. 
\end{proof}

We set 
$$M_c(\underbar{E}_p,d):=Cone(M(\underbar{E}_p,d)
\buildrel{i_.^*}\over\to M(\underbar{E}_{p,|Z},d))[1].$$

We have $M_c(\underbar{E}_p,d)\in D^b_c(R)$ and we can show exactly as in \cite{MR}, Lem. 6.11

\begin{lemma}
For $d$ large enough, we have 
$$H^i(sM_c(\underbar{E}_p,d))\otimes\QQ_p\simeq H^i_{rig,c}(V,E_p).$$
\end{lemma}

Let $r$ be an integer and $d$ large enough. We set 

$$H^i_{abs}(V,\underbar{E}_p(r)):=Ext^i(\underbar{W},M_c(\underbar{E}_p,d)(r))$$
 
$$\chi(\underbar{E}_p,r):=\chi(\underbar{W}, M_c(\underbar{E}_p,d)(r)),$$
if it is well-defined.

Finally we set $\tilde{\chi}(\underbar{E}_p,r):=\tilde{\chi}(P,r)$, with $P=\RR\underbar{Hom}(\underbar{W}, M_c(\underbar{E}_p,d))$. We deduce from the previous lemma and from Th. \ref{MR-main}:

\begin{corollary} Let $\underbar{E}_p$ be as before. Assume that $q^r$ is not a multiple root of the minimum polynomial of $F^a$ acting on $H^j_{rig,c}(V,E_p)$ for any $j$.
\begin{enumerate}
\item The groups $H^i_{abs}(V,\underbar{E}_p(r))$ are finitely generated $\ZZ_p$-modules and the alternating sum of their rank is zero. 

\item The $L-$function $L(V,E,t)$  has a pole at $t=q^{-r}$ of order 
$$\rho:=\sum_j (-1)^{j+1}j.rank_{\ZZ_p}(H^i_{abs}(V,\underbar{E}_p(r))).$$

\item  $\chi(\underbar{E}_p,r)$ is well-defined and we have 
$$| lim_{t \to q^{-r}}L(V,E,t) (1-q^rt)^\rho|_p^{-1}=\chi(\underbar{E}_p,r).q^{\tilde{\chi}(\underbar{E}_p,r)}$$
\end{enumerate}
\end{corollary}

\subsection{The $\ell$-adic valuation of the leading coefficient}

This section is not new but merely a rewriting (and a corollary) of \cite{MR2}, section 4. 

Let $\ell\neq p$ be a prime. We denote $D^b_c(k,\ZZ_\ell)$ the derived category of bounded constructible $\ZZ_\ell$-complexes on $Spec(k)$ (\cite{Ek2}) and $\ZZ_\ell\in D^b_c(k,\ZZ_l)$ will denote the constant object. We have an equivalence of categories 
$$D^b_c(k,\ZZ_\ell)\simeq D^b_c(k,\ZZ/\ell^.),$$
where the left hand side is the derived category of projective systems $M=(M_n)_n$, where $M_n$ is a complex of $\ZZ/\ell^n$ sheaf on $Spec(k)$ and $M$ is such that the projective system 
$$(\ZZ/\ell\otimes^L_{\ZZ/\ell^n} M_n)_n$$
is isomorphic to the constant projective system defined by an object of $D^b_c(k,\ZZ/\ell)$ modulo a projective system of complexes $K=(K_n)_n$ such that the transition maps of the projective system $(H^i(K_n))_n$ becomes zero for $n$ large enough.

We also denote $D(\ZZ_\ell\Gamma)$ the derived category of complexes of $\ZZ_\ell$-module with continuous action of $\Gamma$ and denote $D^b_c(\ZZ_\ell\Gamma)$ the subcategory of bounded complexes with finitely generated cohomology as $\ZZ_\ell$-module.

We have a canonical functor 
$$\alpha: D^b_c(k,\ZZ_\ell)\simeq D^b_c(k,\ZZ/\ell^.)\to D^b_c(\ZZ_\ell\Gamma)$$
$$M=(M_n)\mapsto \RR\ilim M_n(k^{sep}).$$ 

As in the case of $D^b_c(R)$ we have an internal Hom, $\RR\underbar{Hom}(.,.)$ and we can define as before the complex $Ext^.(M,N(r))$, for any $M,N\in D^b_c(k,\ZZ_\ell)$ and integer $r$.

The following theorem was proved in \cite{MR2}

\begin{theorem}\label{mainth_l} Let $M,N\in D^b_c(k,\ZZ_\ell)$ and let $P=\RR\underbar{Hom}(M,N)$. Assume that for all $j$, the minimal polynomial of $\gamma$ acting on $H^j(sP)\otimes\QQ_\ell$ does not have $q^r$ as multiple root.
\begin{enumerate}
\item The groups $Ext^j(M,N(r))$ are finitely generated $\ZZ_\ell$-modules and the alternating sum of their ranks is zero.
\item The Zeta function $Z(P,t):=\prod_j det(1-t\gamma, H^j(sP)\otimes\QQ_\ell)^{(-1)^{j+1}}$ has a pole at $t=q^{-r}$ of order 
$$\rho:=\sum_j(-1)^{j+1} j.rank_{\ZZ_\ell} Ext^j(M,N(r)).$$
\item The cohomology groups of the complex $Ext^.(M,N(r))$ and the alternating product of their orders denoted $\chi(M,N,r)$ satisfies
$$|\lim_{t\to q^{-r}}Z(P,t)(1-q^rt)^\rho|_\ell^{-1}=\chi(M,N,r).$$
\end{enumerate}

\end{theorem}

We come back to our strict (and nice) compatible system $E$.

We set

\begin{enumerate}
\item $N(E_l):=\RR\Gamma_{et,c}(V\times_k \bar k,E_l)$.
\item $P=\RR\underbar{Hom}(\ZZ_\ell,N(E_l))$.
\item $H^i_{abs}(V,E_l(r)):=Ext^i(\ZZ_\ell, N(E_l)(r))$.
\item $\chi(E_l,r):=\chi(\ZZ_\ell, N(E_l),r)$. 
\end{enumerate}

We deduce from Theorem \ref{mainth_l}:

\begin{corollary} Assume that $q^r$ is not a multiple root of the minimum polynomial of $\varphi_l$ acting on $H^j_{et,c}(V\times_k \bar k,E_l)$ for any $j$.
\begin{enumerate}
\item The groups $H^i_{abs}(V,E_\ell(r))$ are finitely generated $\ZZ_\ell$-modules and the alternating sum of their rank is zero. 

\item The $L-$function $L(V,E,t)$  has a pole at $t=q^{-r}$ of order 
$$\rho:=\sum_j (-1)^{j+1}j.rank_{\ZZ_\ell}(H^i_{abs}(V,E_\ell(r))).$$

\item  $\chi(E_\ell,r)$ is well-defined and we have 
$$| lim_{t \to q^{-r}}L(V,E,t) (1-q^rt)^\rho|_\ell^{-1}=\chi(E_l,r).$$
\end{enumerate}
\end{corollary}
\end{section}

\end{document}